\UseRawInputEncoding

\documentclass[oneside, 11pt]{amsart} 

\usepackage{amsmath,amsthm,amssymb,epic}
\usepackage{eqlist,eqparbox}
\usepackage{amsfonts}
\usepackage{latexsym}
\usepackage[2emode]{psfrag}
\usepackage{amsthm}
\usepackage{amsmath}
\usepackage[all]{xy}

\newcommand{\Title}[1]{\bigskip\bigskip\centerline{\bf #1}\bigskip}
\newcommand{\Author}[1]{\medskip\centerline{ \it #1}}

\newcommand{\Affiliation}[1]{\medskip\centerline{#1}}
\newcommand{\Email}[1]{\medskip\centerline{#1}\bigskip}

\begin{document}

\newcommand{\N}{\mbox {$\mathbb N $}}
\newcommand{\Z}{\mbox {$\mathbb Z $}}
\newcommand{\Q}{\mbox {$\mathbb Q $}}
\newcommand{\R}{\mbox {$\mathbb R $}}
\newcommand{\lo }{\longrightarrow }
\newcommand{\ul}{\underleftarrow }
\newcommand{\rl}{\underrightarrow }
\newcommand{\rs }{\rightsquigarrow }
\newcommand{\ra }{\rightarrow }
\newcommand{\dd }{\rightsquigarrow }
\newcommand{\ol }{\overline }
\newcommand{\la }{\langle }
\newcommand{\tr }{\triangle }
\newcommand{\xr }{\xrightarrow }
\newcommand{\de }{\delta }
\newcommand{\pa }{\partial }
\newcommand{\LR }{\Longleftrightarrow }
\newcommand{\Ri }{\Rightarrow }
\newcommand{\va }{\varphi }
\newcommand{\Den}{{\rm Den}\,}
\newcommand{\Ker}{{\rm Ker}\,}
\newcommand{\Reg}{{\rm Reg}\,}
\newcommand{\Fix}{{\rm Fix}\,}
\newcommand{\Sup}{{\rm Sup}\,}
\newcommand{\Inf}{{\rm Inf}\,}
\newcommand{\Img}{{\rm Im}\,}
\newcommand{\Id}{{\rm Id}\,}
\newcommand{\ord}{{\rm ord}\,}

\newtheorem{theorem}{Theorem}[section]
\newtheorem{lemma}[theorem]{Lemma}
\newtheorem{proposition}[theorem]{Proposition}
\newtheorem{corollary}[theorem]{Corollary}
\newtheorem{definition}[theorem]{Definition}
\newtheorem{example}[theorem]{Example}
\newtheorem{examples}[theorem]{Examples}
\newtheorem{xca}[theorem]{Exercise}
\theoremstyle{remark}
\newtheorem{remark}[theorem]{Remark}
\newtheorem{remarks}[theorem]{Remarks}
\numberwithin{equation}{section}

\def\leftmark{L.C. Ciungu}

\Title{IMPLICATIVE-ORTHOMODULAR ALGEBRAS} 
\title[Implicative-orthomodular algebras]{}
                                                                           
\Author{\textbf{LAVINIA CORINA CIUNGU}}
\Affiliation{Department of Mathematics} 
\Affiliation{St Francis College}
\Affiliation{180 Remsen Street, Brooklyn Heights, NY 11201-4398, USA}
\Email{lciungu@sfc.edu}

\begin{abstract} 
Starting from involutive BE algebras, we redefine the orthomodular algebras, by introducing the notion of 
implicative-orthomodular algebras. 
We investigate properties of implicative-orthomodular algebras, and give characterizations of these algebras.   
Then we define and study the notions of filters and deductive systems, and characterize certain classes of filters. 
Furthermore, we introduce and characterize the commutative deductive systems in implicative-orthomodular algebras. 
We also show that any deductive system determines a congruence, and conversely, for any congruence we can define a deductive system in an implicative-orthomodular algebra. 
We define the quotient implicative-orthomodular algebra with respect to the congruence induced by a deductive system, 
and prove that a deductive system is commutative if and only if all deductive systems of the corresponding quotient algebra are commutative. \\

\textbf{AMS classification (2020):} 06F35, 03G25, 06A06, 81P10, 06C15
\end{abstract}

\maketitle

\section{Introduction} 

In the last decades, the study of algebraic structures related to the logical foundations of quantum mechanics 
became a central topic of research. Generally known as quantum structures, these algebras serve as algebraic 
semantics for the classical and non-classical logics, as well as for the quantum logics. 
As algebraic structures connected with quantum logics we mention the following algebras: bounded involutive lattices,  De Morgan algebras, ortholattices, orthomodular lattices, MV algebras, quantum MV algebras, orthomodular algebras.  
The quantum-MV algebras (or QMV algebras) were introduced by R. Giuntini in \cite{Giunt1} 
as non-lattice generalizations of MV algebras and as non-idempotent generalizations of orthomodular lattices. 
These structures were intensively studied by R. Giuntini (\cite{Giunt2, Giunt3, Giunt4, Giunt5, Giunt6}), 
A. Dvure\v censkij and S. Pulmannov\'a (\cite{DvPu}), R. Giuntini and S. Pulmannov\'a (\cite{Giunt7}) and by 
A. Iorgulescu in \cite{Ior30, Ior31, Ior32, Ior33, Ior34, Ior35}. 
An extensive study on the orthomodular structures as quantum logics can be found in \cite{Ptak}. \\
A. Iorgulescu introduced and studied the m-BE algebras (\cite{Ior30}) and the orthomodular algebras as 
generalizations of quantum-MV algebras (\cite{Ior31}). 
A left-m-BE algebra is an algebra $(X,\odot,^{*},1)$ of type $(2,1,0)$ satisfying the following properties: 
(PU) $1\odot x=x=x\odot 1$, (Pcomm) $x\odot y=y\odot x$, (Pass) $x\odot (y\odot z)=(x\odot y)\odot z$, 
(m-L) $x\odot 0=0$, (m-Re) $x\odot x^{*}=0$, where $0:=1^*$. 
A left-orthomodular algebra (left-OM algebra for short) is an involutive left-m-BE algebra satisfying 
condition (Pom) $(x\odot y)\oplus ((x\odot y)^*\odot x)=x$, where $x\oplus y=(x^*\odot y^*)^*$. 
The right-m-BE algebras and right-orthomodular algebras (right-OM algebras for short) are dually defined. 
The orthomodular algebras (OM algebras for short) were defined as left-OM algebras. 
We redefined in \cite{Ciu78} the quantum-MV algebras as involutive BE algebras, by introducing the notion of quantum-Wajsberg algebras (or QW algebras for short) and studied these structures. 
The implicative-ortholattices (also particular involutive BE algebras, introduced and studied in 
\cite{Ior30, Ior35}) are definitionally equivalent to ortholattices (redefined as particular involutive m-BE algebras in \cite{Ior30, Ior35}), and the implicative-Boolean algebras (introduced by A. Iorgulescu in 2009, also as particular involutive BE algebras, namely as particular involutive BCK algebras) are definitionally equivalent to Boolean algebras (redefined as particular involutive m-BE algebras, namely as particular (involutive) m-BCK algebras in \cite{Ior35}). \\
In this paper, we redefine the orthomodular algebras, by introducing the notion of implicative-orthomodular algebras, and we show that the almost all the properties verified by quantum-Wajsberg algebras are also verified by implicative-orthomodular algebras. 
We prove that any quantum-Wajsberg algebra is an implicative-orthomodular algebra, but the converse is not true. 
We characterize the implicative-orthomodular algebras and we give conditions for implicative-orthomodular algebras to be quantum-Wajsberg algebras. 
We also introduce and study the notions of filters and deductive systems in implicative-orthomodular algebras, as well as the maximal and strongly maximal filters, and we show that a strongly maximal filter is maximal. 
Furthermore, we define and characterize the commutative deductive systems, and we prove that any 
deductive system of a quantum-Wajsberg algebra is commutative. 
We also show that any deductive system determines a congruence, and conversely, for any congruence we can define a deductive system in an implicative-orthomodular algebra. 
We define the quotient implicative-orthomodular algebra with respect to the congruence induced by a 
deductive system, and prove that a deductive system is commutative if and only if all deductive 
systems of the corresponding quotient algebra are commutative. 
We mention that the filters and congruences in quantum-Wajsberg algebras have been studied in \cite{Ciu79}, and we will 
show that certain results proved for quantum-Wajsberg algebras are also valid in the case of implicative-orthomodular algebras. 
Additionally, we prove new properties of involutive BE algebras.

$\vspace*{1mm}$

\section{Preliminaries}

In this section, we recall some basic notions and results regarding BE algebras and quantum-Wajsberg algebras   
that will be used in the paper. Additionally, we prove new properties of involutive BE algebras. 
For more details on quantum-Wajsberg algebras we refer the reader to \cite{Ciu78}. 

The \emph{BE algebras} were introduced in \cite{Kim1} as algebras $(X,\ra,1)$ of type $(2,0)$ satisfying the 
following conditions, for all $x,y,z\in X$: 
$(BE_1)$ $x\ra x=1;$ 
$(BE_2)$ $x\ra 1=1;$ 
$(BE_3)$ $1\ra x=x;$ 
$(BE_4)$ $x\ra (y\ra z)=y\ra (x\ra z)$. 
A relation $\le$ is defined on $X$ by $x\le y$ iff $x\ra y=1$. 
A BE algebra $X$ is \emph{bounded} if there exists $0\in X$ such that $0\le x$, for all $x\in X$. 
In a bounded BE algebra $(X,\ra,0,1)$ we define $x^*=x\ra 0$, for all $x\in X$. 
A bounded BE algebra $X$ is called \emph{involutive} if $x^{**}=x$, for any $x\in X$. \\
Note that, according to \cite[Cor. 17.1.3]{Ior35}, the involutive BE algebras are definitionally equivalent to involutive m-BE algebras, by the mutually inverse transformations (\cite{Ior30, Ior35}): \\ 
$\hspace*{3cm}$ $\Phi:$\hspace*{0.2cm}$ x\odot y:=(x\ra y^*)^*$ $\hspace*{0.1cm}$ and  
                $\hspace*{0.1cm}$ $\Psi:$\hspace*{0.2cm}$ x\ra y:=(x\odot y^*)^*$.

\begin{lemma} \label{qbe-10} $\rm($\cite{Ciu78}$\rm)$ 
Let $(X,\ra,1)$ be a BE algebra. The following hold for all $x,y,z\in X$: \\
$(1)$ $x\ra (y\ra x)=1;$ \\
$(2)$ $x\le (x\ra y)\ra y$. \\
If $X$ is bounded, then: \\
$(3)$ $x\ra y^*=y\ra x^*;$ \\
$(4)$ $x\le x^{**}$. \\
If $X$ is involutive, then: \\
$(5)$ $x^*\ra y=y^*\ra x;$ \\
$(6)$ $x^*\ra y^*=y\ra x;$ \\
$(7)$ $(x\ra y)^*\ra z=x\ra (y^*\ra z);$ \\
$(8)$ $x\ra (y\ra z)=(x\ra y^*)^*\ra z;$ \\   
$(9)$ $(x^*\ra y)^*\ra (x^*\ra y)=(x^*\ra x)^*\ra (y^*\ra y)$.  
\end{lemma}

\noindent
In a BE algebra $X$, we define the additional operation: \\
$\hspace*{3cm}$ $x\Cup y=(x\ra y)\ra y$. \\
If $X$ is involutive, we define the operations: \\
$\hspace*{3cm}$ $x\Cap y=((x^*\ra y^*)\ra y^*)^*$, \\
$\hspace*{3cm}$ $x\odot y=(x\ra y^*)^*=(y\ra x^*)^*$, \\
and the relation $\le_Q$ by: \\
$\hspace*{3cm}$ $x\le_Q y$ iff $x=x\Cap y$. \\
We can easily check that the operation $\odot$ is commutative and associative. 

\begin{proposition} \label{qbe-20} $\rm($\cite{Ciu78}$\rm)$ Let $X$ be an involutive BE algebra. 
Then the following hold for all $x,y,z\in X$: \\
$(1)$ $x\le_Q y$ implies $x=y\Cap x$ and $y=x\Cup y;$ \\
$(2)$ $\le_Q$ is reflexive and antisymmetric; \\
$(3)$ $x\Cap y=(x^*\Cup y^*)^*$ and $x\Cup y=(x^*\Cap y^*)^*;$ \\ 
$(4)$ $x\le_Q y$ implies $x\le y;$ \\
$(5)$ $0\le_Q x \le_Q 1;$ \\
$(6)$ $0\Cap x=x\Cap 0=0$ and $1\Cap x=x\Cap 1=x;$ \\
$(7)$ $(x\Cap y)\ra z=(y\ra x)\ra (y\ra z);$ \\
$(8)$ $z\ra (x\Cup y)=(x\ra y)\ra (z\ra y);$ \\
$(9)$ $x\Cap y\le x,y\le x\Cup y;$ \\
$(10)$ $x\Cap (y\Cap x)=y\Cap x$ and $x\Cap (x\Cap y)=x\Cap y$. 
\end{proposition}

\begin{proposition} \label{qbe-30} Let $X$ be an involutive BE algebra. 
Then the following hold for all $x,y,z\in X$: \\
$(1)$ $x, y\le_Q z$ and $z\ra x=z\ra y$ imply $x=y;$ \emph{(cancellation law)} \\   
$(2)$ $(x\ra (y\ra z))\ra x^*=((y\ra z)\Cap x)^*;$ \\
$(3)$ $x\ra ((y\ra x^*)^*\Cup z)=y\Cup (x\ra z);$ \\
$(4)$ $((y\ra x)\Cap z)\ra x=y\Cup (z\ra x);$ \\
$(5)$ $x\le_Q y$ implies $(y\ra x)\odot y=x;$ \\ 
$(6)$ $x\ra (z\odot y^*)=((z\ra y)\odot x)^*;$ \\
$(7)$ $(x\Cup y)\Cap y=y$ and $(x\Cap y)\Cup y=y$.   
\end{proposition}
\begin{proof} 
$(1)$ Since $x,y\le_Q z$ and $z\ra x=z\ra y$, we have: \\
$\hspace*{2.00cm}$ $x=x\Cap z=((x^*\ra z^*)\ra z^*)^*=((z\ra x)\ra z^*)^*$ \\
$\hspace*{2.30cm}$ $=((z\ra y)\ra z^*)^*=((y^*\ra z^*)\ra z^*)^*=y\Cap z=y$. \\
$(2)$ Applying Lemma \ref{qbe-10}, we get: \\
$\hspace*{1.00cm}$ $((y\ra z)\Cap x)^*=(y\ra z)^*\Cup x^*=((y\ra z)^*\ra x^*)\ra x^*=(x\ra (y\ra z))\ra x^*$. \\
$(3)$ Using Lemma \ref{qbe-10}$(8)$, we have \\
$\hspace*{1.00cm}$ $x\ra ((y\ra x^*)^*\Cup z)=x\ra (((y\ra x^*)^*\ra z)\ra z)$ \\
$\hspace*{4.50cm}$ $=((y\ra x^*)^*\ra z)\ra (x\ra z)$ \\
$\hspace*{4.50cm}$ $=(y\ra (x\ra z))\ra (x\ra z)=y\Cup (x\ra z)$. \\ 
$(4)$ Similarly we have: \\
$\hspace*{1.00cm}$ $((y\ra x)\Cap z)\ra x=x^*\ra ((y\ra x)^*\Cup z^*)=x^*\ra (((y\ra x)^*\ra z^*)\ra z^*)$ \\
$\hspace*{4.00cm}$ $=((y\ra x)^*\ra z^*)\ra (x^*\ra z^*)=(z\ra (y\ra x))\ra (z\ra x)$ \\
$\hspace*{4.00cm}$ $=(y\ra (z\ra x))\ra (z\ra x)=y\Cup (z\ra x)$. \\
$(5)$ Since $x\le_Q y$, we have: \\
$\hspace*{2cm}$ $(y\ra x)\odot y=((y\ra x)\ra y^*)^*=((x^*\ra y^*)\ra y^*)^*=x\Cap y=x$. \\
$(6)$ Taking into consideration that $x\ra y=(x\odot y^*)^*$, we get: \\
$\hspace*{2cm}$ $x\ra (z\odot y^*)=(x\odot (z\odot y^*)^*)^*=(x\odot (z\ra y))^*$. \\
$(7)$ We have: \\ 
$\hspace*{2.00cm}$ $(x\Cup y)\Cap y=(((x\Cup y)^*\ra y^*)\ra y^*)^*$ \\
$\hspace*{4.00cm}$ $=((y\ra (x\Cup y))\ra y^*)^*$ \\
$\hspace*{4.00cm}$ $=((y\ra ((x\ra y)\ra y))\ra y^*)^*$ \\
$\hspace*{4.00cm}$ $=(((x\ra y)\ra (y\ra y))\ra y^*)^*=(y^*)^*=y$. \\
Similarly $(x\Cap y)\Cup y=y$.    
\end{proof}

\noindent 
\emph{Quantum-Wajsberg algebras} (\emph{QW algebras, for short}) were defined and studied in \cite{Ciu78} 
as involutive BE algebras $X$ satisfying the following axiom for all $x,y,z\in X$: \\
(QW) $x\ra ((x\Cap y)\Cap (z\Cap x))=(x\ra y)\Cap (x\ra z)$. 

\noindent
Condition (QW) is equivalent to the following conditions: \\
($QW_1$) $x\ra (x\Cap y)=x\ra y;$ \\ 
($QW_2$) $x\ra (y\Cap (z\Cap x))=(x\ra y)\Cap (x\ra z)$. 

\begin{proposition} \label{qbe-60} $\rm($\cite{Ciu78}$\rm)$ Let $X$ be a quantum-Wajsberg algebra. 
The following hold for all $x,y\in X$:\\
$(1)$ $x\ra (y\Cap x)=x\ra y$ and $(x\ra y)\ra (y\Cap x)=x;$ \\
$(2)$ $x\le_Q x^*\ra y$ and $x\le_Q y\ra x;$ \\
$(3)$ $x\ra y=0$ iff $x=1$ and $y=0;$ \\
$(4)$ $(x\ra y)^*\Cap x=(x\ra y)^*;$ \\ 
$(5)$ $(x\Cap y)\Cap y=x\Cap y$ and $(x\Cup y)\Cup y=x\Cup y;$ \\
$(6)$ $x\Cup (y\Cap x)=x$ and $x\Cap (y\Cup x)=x;$ \\
$(7)$ $x\Cap y\le_Q y\le_Q x\Cup y;$ \\
$(8)$ $(x\Cup y)\ra x=(y\Cup x)\ra x=y\ra x;$ \\
$(9)$ $(x\Cup y)\ra y=(y\Cup x)\ra y=x\ra y;$ \\
$(10)$ $x\le y$ iff $y\Cap x=x$.      
\end{proposition}

\begin{proposition} \label{qbe-70} $\rm($\cite{Ciu78}$\rm)$ Let $X$ be a quantum-Wajsberg algebra. 
If $x,y\in X$ such that $x\le_Q y$, then the following hold for any $z\in X$:\\
$(1)$ $y=y\Cup x;$ \\
$(2)$ $y^*\le_Q x^*;$ \\
$(3)$ $y\ra z\le_Q x\ra z$ and $z\ra x\le_Q z\ra y;$ \\
$(4)$ $x\Cap z\le_Q y\Cap z$ and $x\Cup z\le_Q y\Cup z$. 
\end{proposition}

\begin{proposition} \label{qbe-80} $\rm($\cite{Ciu78}$\rm)$ Let $X$ be a quantum-Wajsberg algebra. 
The following hold, for all $x,y,z\in X$:\\
$(1)$ $(x\Cap y)\Cap (y\Cap z)=(x\Cap y)\Cap z;$ \\
$(2)$ $\le_Q$ is transitive; \\
$(3)$ $x\Cup y\le_Q x^*\ra y;$ \\
$(4)$ $(x^*\ra y)^*\ra (x\ra y^*)^*=x^*\ra y;$ \\
$(5)$ $(x\ra y)^*\ra (y\ra x)^*=x\ra y;$ \\
$(6)$ $(y\ra x)\ra (x\ra y)=x\ra y;$ \\
$(7)$ $(x\ra y)\Cup (y\ra x)=1;$ \\
$(8)$ $(z\Cap x)\ra (y\Cap x)=(z\Cap x)\ra y$.    
\end{proposition}

\noindent
By Propositions \ref{qbe-20}$(2)$, \ref{qbe-80}$(2)$, in a quantum-Wajsberg algebra $X$, $\le_Q$ is a partial 
order on $X$. 

\begin{proposition} \label{qbe-120} $\rm($\cite{Ciu78}$\rm)$   
In any quantum-Wajsberg algebra $X$ the following hold for all $x,y,z\in X$: \\
$(1)$ $x\ra (y\ra z)=(x\odot y)\ra z;$ \\
$(2)$ $x\le_Q y\ra z$ implies $x\odot y\le z;$ \\
$(3)$ $x\odot y\le z$ implies $x\le y\ra z;$ \\
$(4)$ $(x\ra y)\odot x\le y;$ \\
$(5)$ $x\le_Q y$ implies $x\odot z\le_Q y\odot z;$ \\
$(6)$ $x\le_Q y$ implies $(y\ra x)\odot y=x;$ \\ 
$(7)$ $x\ra (z\odot y^*)=((z\ra y)\odot x)^*$. 
\end{proposition}

$\vspace*{1mm}$

\section{Implicative-orthomodular algebras}

In this section, we redefine the orthomodular algebras, by introducing the notion of implicative-orthomodular algebras, and 
we study the properties of these algebras. 
We show that the almost all the properties verified by quantum-Wajsberg algebras are also verified by implicative-orthomodular algebras. 
Furthermore, we prove that any quantum-Wajsberg algebra is an implicative-orthomodular algebra, but the converse is not true. 
We also characterize these algebras and we give conditions for implicative-orthomodular algebras to be quantum-Wajsberg algebras. 

\begin{lemma} \label{roma-10} Let $X$ be an involutive BE algebra. 
The following are equivalent for all $x,y\in X$: \\
$(IOM)$ $x\Cap (y\ra x)=x;$ \\
$(IOM^{'})$ $x\Cap (x^*\ra y)=x;$ $\hspace*{5cm}$ \\
$(IOM^{''})$ $x\Cup (x\ra y)^*=x$. 
\end{lemma}
\begin{proof}
The proof is straightforward. 
\end{proof}

\begin{definition} \label{roma-20} \emph {
An involutive BE algebra $X$ satisfying the equivalent conditions from Lemma \ref{roma-10} is called an 
\emph{implicative-orthomodular algebra} (\emph{IOM algebra}, for short). 
}
\end{definition}

\begin{remark} \label{roma-20-10}
By Proposition \ref{qbe-60}$(2)$, any QW algebra is an IOM algebra. 
\end{remark}

\begin{proposition} \label{roma-20-20} Implicative-orthomodular algebras are definitionally equivalent to orthomodular algebras. 
\end{proposition}
\begin{proof} By definitions of $\odot$ and $\oplus$, and applying Lemma \ref{qbe-10} we get: \\
$\hspace*{2cm}$ $(x\odot y)\oplus ((x\odot y)^*\odot x)=x$ iff \\
$\hspace*{2cm}$ $(x\ra y^*)^*\oplus ((x\ra y^*)\ra x^*)^*=x$ iff \\
$\hspace*{2cm}$ $(x\ra y^*)\ra ((x\ra y^*)\ra x^*)^*=x$ iff \\
$\hspace*{2cm}$ $((x\ra y^*)\ra x^*)\ra (x\ra y^*)^*=x$ iff \\
$\hspace*{2cm}$ $(x\ra (x\ra y^*)^*)\ra (x\ra y^*)^*=x$ iff \\
$\hspace*{2cm}$ $x^*\Cap (y\ra x^*)=x^*$, \\ 
for all $x,y\in X$. 
Since $X^*=\{x^*\mid x\in X\}=X$, it follows that axiom $x^*\Cap (y\ra x^*)=x^*$, for all $x,y\in X$ is 
equivalent to axiom $x\Cap (y\ra x)=x$, for all $x,y\in X$. 
Hence axioms (Pom) and (IOM) are equivalent, so that the implicative-orthomodular algebras are definitionally equivalent to orthomodular algebras. 
\end{proof}

\begin{proposition} \label{roma-30} Let $X$ be an IOM algebra. The following hold for all $x,y,z\in X$: \\
$(1)$ $x\Cap (y\Cup x)=x$ and $x\Cup (y\Cap x)=x$. \\
If $x\le_Q y$, then: \\
$(2)$ $y\Cup x=y$ and $y^*\le_Q x^*;$ \\ 
$(3$ $y\ra z\le_Q x\ra z$ and $z\ra x\le_Q z\ra y;$ \\
$(4)$ $x\Cap z\le_Q y\Cap z$ and $x\Cup z\le_Q y\Cup z;$ \\
$(5)$ $(z\ra y)\Cup (z\ra x)=z\ra y$. 
\end{proposition}
\begin{proof}
$(1)$ By $(IOM)$, we have: \\
$\hspace*{1.50cm}$ $x\Cap (y\Cup x)=x\Cap ((y\ra x)\ra x)=x;$ \\ 
$\hspace*{1.50cm}$ $x\Cup (y\Cap x)=x\Cup((y^*\ra x^*)\ra x^*)^*=(x^*\Cap ((y\ra x^*)\ra x^*))^*=(x^*)^*=x$. \\
$(2)$ Since $x=x\Cap y$, using $(1)$ we get $y\Cup x=y\Cup (x\Cap y)=y$. 
It follows that $y^*=(y\Cup x)^*=y^*\Cap x^*$, hence $y^*\le_Q x^*$. \\
$(3)$ Since by $(IOM^{'})$, $y\ra z\le_Q (y\ra z)^*\ra (y\ra x)^*$,  we have: \\
$\hspace*{2.00cm}$ $(y\ra z)\Cap (x\ra z)=(y\ra z)\Cap ((x\Cap y)\ra z)$ \\
$\hspace*{5.15cm}$ $=(y\ra z)\Cap(((x^*\ra y^*)\ra y^*)^*\ra z)$ \\
$\hspace*{5.15cm}$ $=(y\ra z)\Cap (z^*\ra ((x^*\ra y^*)\ra y^*))$ \\
$\hspace*{5.15cm}$ $=(y\ra z)\Cap ((x^*\ra y^*)\ra (z^*\ra y^*))$ \\
$\hspace*{5.15cm}$ $=(y\ra z)\Cap ((y\ra x)\ra (y\ra z))$ \\
$\hspace*{5.15cm}$ $=(y\ra z)\Cap ((y\ra z)^*\ra (y\ra x)^*)$ \\
$\hspace*{5.15cm}$ $=y\ra z$. \\
It follows that $y\ra z\le_Q x\ra z$. 
By $(2)$, $x\le_Q y$ implies $y=y\Cup x$, and by $(IOM^{'})$ 
$z\ra x\le_Q (z\ra x)^*\ra (y\ra x)^*$, hence: \\
$\hspace*{2.00cm}$ $(z\ra x)\Cap (z\ra y)=(z\ra x)\Cap (z\ra (y\Cup x))$ \\
$\hspace*{5.15cm}$ $=(z\ra x)\Cap (z\ra ((y\ra x)\ra x))$ \\
$\hspace*{5.15cm}$ $=(z\ra x)\Cap ((y\ra x)\ra (z\ra x))$ \\
$\hspace*{5.15cm}$ $=(z\ra x)\Cap ((z\ra x)^*\ra (y\ra x)^*)$ \\
$\hspace*{5.15cm}$ $=z\ra x$. \\
Thus $z\ra x\le_Q z\ra y$. \\
$(4)$ Since $x\le_Q y$ implies $y^*\le_Q x^*$, using $(3)$ we get $x^*\ra z^*\le_Q y^*\ra z^*$ and 
$(y^*\ra z^*)\ra z^*\le_Q (x^*\ra z^*)\ra z^*$. 
It follows that $((x^*\ra z^*)\ra z^*)^*\le_Q ((y^*\ra z^*)\ra z^*)^*$, 
that is $x\Cap z\le_Q y\Cap z$. 
Applying again $(3)$, we also have $y\ra z\le_Q x\ra z$ and $(x\ra z)\ra z\le_Q (y\ra z)\ra z$. 
Hence $x\Cup z\le_Q y\Cup z$. \\
$(5)$ Since by $(3)$, $x\le_Q y$ implies $z\ra x\le_Q z\ra y$, using $(2)$ we get $(z\ra y)\Cup (z\ra x)=z\ra y$. 
\end{proof}

\begin{proposition} \label{roma-40} Let $X$ be an IOM algebra. The following hold for all $x,y,z\in X$: \\
$(1)$ $(x\ra y)\Cup y=x\ra y;$ \\
$(2)$ $(x\ra y)\ra (y\Cap x)=x;$ \\
$(3)$ $x\ra (y\Cap x)=x\ra y;$ \\
$(4)$ $(x\Cup y)\ra (x\ra y)^*=y^*;$ \\ 
$(5)$ $x\Cap ((y\ra x)\Cap (z\ra x))=x;$ \\
$(6)$ $x\le y$ iff $y\Cap x=x;$ \\
$(7)$ $x\le_Q y$ and $y\le x$ imply $x=y;$ \\
$(8)$ $x\Cap y\le_Q y\le_Q x\Cup y;$ \\
$(9)$ $(x\Cup y)\ra y=x\ra y;$ \\
$(10)$ $x\Cap y, y\Cap x\le_Q x\ra y$.      
\end{proposition}
\begin{proof}
$(1)$ Applying (IOM) and Proposition \ref{roma-30}$(1)$, we get \\
$\hspace*{2cm}$ $(x\ra y)\Cup y=(x\ra y)\Cup (y\Cap (x\ra y))=x\ra y$. \\
$(2)$ Using (IOM), we have: \\
$\hspace*{2.00cm}$ $x^*=(x\Cap (y\ra x))^*=(x^*\ra (y\ra x)^*)\ra (y\ra x)^*$ \\
$\hspace*{2.50cm}$ $=((y\ra x)\ra x)\ra (y\ra x)^*=(y\Cup x)\ra (y\ra x)^*$ \\
$\hspace*{2.50cm}$ $=(y\ra x)\ra (y\Cup x)^*=(x^*\ra y^*)\ra (y^*\Cap x^*)$. \\
Replacing $x$ by $x^*$ and $y$ by $y^*$, we get $(x\ra y)\ra (y\Cap x)=x$. \\
$(3)$ Since $x^*\le_Q y^*\ra x^*$, applying \ref{roma-30}$(2)$, we get: \\
$\hspace*{2.00cm}$ $x\ra (y\Cap x)=x\ra ((y^*\ra x^*)\ra x^*)^*=((y^*\ra x^*)\ra x^*)\ra x^*$ \\
$\hspace*{4.00cm}$ $=(y^*\ra x^*)\Cup x^*=y^*\ra x^*=x\ra y$. \\
$(4)$ Applying (IOM), we have $y=y\Cap (x\ra y)=((y^*\ra (x\ra y)^*)\ra (x\ra y)^*)^*$. 
It follows that $y^*=(y^*\ra (x\ra y)^*)\ra (x\ra y)^*=((x\ra y)\ra y)\ra (x\ra y)^*=(x\Cup y)\ra (x\ra y)^*$. \\
$(5)$ Since $x\le_Q y\ra x, z\ra x$, applying Proposition \ref{roma-30}$(4)$, we have: 
$x=x\Cap (z\ra x)\le_Q (y\ra x)\Cap (z\ra x)$, hence $x\Cap ((y\ra x)\Cap (z\ra x))=x$. \\
$(6)$ If $x\le y$, then $x\ra y=1$ and we have 
$y\Cap x=((y^*\ra x^*)\ra x^*)^*=((x\ra y)\ra x^*)^*=(1\ra x^*)^*=(x^*)^*=x$. 
Conversely, if $y\Cap x=x$, using $(3)$ we get $x\ra y=x\ra (y\Cap x)=x\ra x=1$, hence $x\le y$. \\
$(7)$ By $(6)$, $y\le x$ implies $x\Cap y=y$. Since $x\le_Q y$, we have $x\Cap y=x$, hence $x=y$. \\
$(8)$ It follows by $(IOM)$ and Proposition \ref{roma-30}$(2)$. \\
$(9)$ Using $(3)$, we get: 
$(x\Cup y)\ra y=y^*\ra (x\Cup y)^*=y^*\ra (x^*\Cap y^*)=y^*\ra x^*=x\ra y$. \\
$(10)$ Using $(IOM)$, we have $y\le_Q x\ra y$, so that 
$(x\ra y)^*\le_Q y^*\le_Q (x^*\ra y^*)\ra y^*=(x\Cap y)^*$. Hence $x\Cap y\le_Q x\ra y$. 
Similarly $(x\ra y)^*\le (x\ra y)\ra x^*=(y^*\ra x^*)\ra x^*=(y\Cap x)^*$. 
Thus $y\Cap x\le_Q x\ra y$. 
\end{proof}

\begin{theorem} \label{roma-50} An involutive BE algebra $X$ is an IOM algebra if and only if it satisfies  
axiom $(QW_2)$. 
\end{theorem}
\begin{proof}
Taking $y:=0$ in $(QW_2)$, we have $x^*=x^*\Cap (x\ra z)$, and replacing $x$ with $x^*$ and $z$ with $y$,  
we get $x=x\Cap (x^*\ra y)$. It follows that $X$ satisfies condition $(IOM^{'}$), hence it is 
an IOM algebra. 
Conversely, suppose that $X$ is an IOM algebra. 
From Proposition \ref{roma-40}$(5)$, we have $x^*\Cup ((y\ra x)^*\Cup (z\ra x)^*))=x^*$. 
Replacing $x, y, z$ with $x^*, y^*, z^*$, respectively, we get $x\Cup ((y^*\ra x^*)^*\Cup (z^*\ra x^*))=x$. 
It follows that: \\
$\hspace*{1.00cm}$ $x\ra (y\Cap (z\Cap x))=x\ra (y\Cap (z^*\Cup x^*)^*)$ \\
$\hspace*{4.00cm}$ $=x\ra (y^*\Cup ((z^*\Cup x^*))^*$ \\ 
$\hspace*{4.00cm}$ $=x\ra (y^*\Cup ((z^*\ra x^*)\ra x^*))^*$ \\
$\hspace*{4.00cm}$ $=x\ra (y^*\Cup (x\ra (z^*\ra x^*)^*)^*$ \\
$\hspace*{4.00cm}$ $=x\ra (x\ra ((y^*\ra x^*)^*\Cup (z^*\ra x^*)^*))^*$ (Prop. \ref{qbe-30}$(3)$) \\
$\hspace*{4.00cm}$ $=(x\Cup ((y^*\ra x^*)^*\Cup (z^*\ra x^*)^*))\ra (x\ra ((y^*\ra x^*)^*\Cup (z^*\ra x^*)^*))^*$ \
$\hspace*{4.00cm}$ $=((y^*\ra x^*)^*\Cup (z^*\ra x^*)^*)^*$ (Prop. \ref{roma-40}$(4)$) \\
$\hspace*{4.00cm}$ $=(y^*\ra x^*)\Cap (z^*\ra x^*)=(x\ra y)\Cap (x\ra z)$. \\
Hence axiom $(QW_2)$ is satisfied. 
\end{proof}



\begin{proposition} \label{roma-80} Let $X$ be an IOM algebra. The following hold for all $x,y,z\in X$: \\
$(1)$ $(x\Cap y)\Cap y=x\Cap y;$ \\
$(2)$ $x\Cup (y\Cap x)=x;$ \\ 
$(3)$ $x\Cap (y\Cup x)=x;$ \\
$(4)$ $x\Cap y\le_Q y\le_Q x\Cup y;$ \\
$(5)$ $(x\Cap y)\Cap (y\Cap z)=(x\Cap y)\Cap z;$ \\
$(6)$ $(x\Cup y)\Cup (y\Cup z)=(x\Cup y)\Cup z;$ \\
$(7)$ $\le_Q$ is transitive; \\
$(8)$ $(x\ra y)\Cup (x\ra (z\Cap y))=x\ra y;$ \\ 
$(9)$ $(x\ra y)\Cup ((z\ra x)\ra y)=x\ra y$.   
\end{proposition}
\begin{proof}
$(1)$ Since $(y\ra (x^*\ra y^*)^*)^*\le_Q y$, we have: \\
$\hspace*{2.00cm}$ $(x\Cap y)\Cap y=((x^*\ra y^*)\ra y^*)^*\Cap y=(y\ra (x^*\ra y^*)^*)^*\Cap y$ \\
$\hspace*{3.85cm}$ $=(y\ra (x^*\ra y^*)^*)^*=((x^*\ra y^*)\ra y^*)^*=x\Cap y$. \\
$(2)$ By Proposition \ref{roma-40}$(3)$, we have: \\
$\hspace*{2.00cm}$ $x=x^{**}=(x^*)^*=(x^*\Cap (x\ra y))^*=x\Cup (x\ra y)^*$ \\
$\hspace*{2.30cm}$ $=(x\ra (x\ra y)^*)\ra (x\ra y)^*$ \\
$\hspace*{2.30cm}$ $=(x\ra y)\ra (x\ra (x\ra y)^*)^*$ \\
$\hspace*{2.30cm}$ $=(x\ra y)\ra ((x\ra y)\ra x^*)^*$ \\
$\hspace*{2.30cm}$ $=(x\ra y)\ra ((y^*\ra x^*)\ra x^*)^*$ \\
$\hspace*{2.30cm}$ $=(x\ra y)\ra (y\Cap x)$ \\
$\hspace*{2.30cm}$ $=(x\ra (y\Cap x))\ra (y\Cap x)$ \\
$\hspace*{2.30cm}$ $=x\Cup (y\Cap x)$. \\
$(3)$ Applying $(2)$, we have: \\
$\hspace*{2.15cm}$ $x\Cap (y\Cup x)=(x^*\Cup (y\Cup x)^*)^*=(x^*\Cup (y^*\Cap x^*))^*=(x^*)^*=x$. \\
$(4)$ By $(1)$, $x\Cap y\le_Q y$. Moreover $y\le_Q (x\ra y)\ra y=x\Cup y$. \\
$(5)$ Using Proposition \ref{qbe-20}$(7)$, since $x\Cap y\le_Q y$, we get: \\
$\hspace*{2.20cm}$ $(x\Cap y)\Cap (y\Cap z)=(((x\Cap y)^*\ra (y\Cap z)^*)\ra (y\Cap z)^*)^*$ \\
$\hspace*{5.00cm}$ $=(((x\Cap y)^*\ra (y^*\Cup z^*))\ra (y^*\Cup z^*))^*$ \\
$\hspace*{5.00cm}$ $=(((x\Cap y)^*\ra ((y^*\ra z^*)\ra z^*))\ra ((y^*\ra z^*)\ra z^*))^*$ \\
$\hspace*{5.00cm}$ $=(((y^*\ra z^*)\ra ((x\Cap y)^*\ra z^*))\ra ((y^*\ra z^*)\ra z^*))^*$ \\
$\hspace*{5.00cm}$ $=((((x\Cap y)^*\ra z^*)\Cap (y^*\ra z^*))\ra z^*)^*$ \\
$\hspace*{5.00cm}$ $=(((x\Cap y)^*\ra z^*)\ra z^*)^*$ \\
$\hspace*{5.00cm}$ $=(x\Cap y)\Cap z$. \\
(Denoting by $a=(x\Cap y)\ra z^*$, $b=y^*\ra z^*$, $c=z^*$, by Proposition \ref{qbe-20}$(7)$ we have 
$(b\ra a)\ra (b\ra c)=(a\Cap b)\ra c$. It follows that: \\
$((y^*\ra z^*)\ra ((x\Cap y)^*\ra z^*))\ra ((y^*\ra z^*)\ra z^*)=(((x\Cap y)^*\ra z^*)\Cap (y^*\ra z^*))\ra z^*$. \\
Since $x\Cap y\le_Q y$, we have $y^*\le_Q (x\Cap y)^*$. 
Hence $(x\Cap y)^*\ra z^*\le_Q y^*\ra z^*$, that is $((x\Cap y)^*\ra z^*)\Cap (y^*\ra z^*)=(x\Cap y)^*\ra z^*$). \\
$(6)$ Applying $(5)$, we have: \\
$\hspace*{2.20cm}$ $(x\Cup y)\Cup (y\Cup z)=((x^*\Cap y^*)\Cap (y^*\Cap z^*)^*)^*=((x^*\Cap y^*)\Cap z^*)^*$ \\
$\hspace*{5.00cm}$ $=(x^*\Cap y^*)^*\Cup z=(x\Cup y)\Cup z$. \\
$(7)$ If $x\le_Q y$ and $y\le_Q z$, then $x=x\Cap y$ and $y=y\Cap z$. 
Applying $(5)$, we get $x=(x\Cap y)\Cap (y\Cap z)=(x\Cap y)\Cap z=x\Cap z$, that is $x\le_Q z$. 
We conclude that $\le_Q$ is transitive. \\
$(8)$ and $(9)$ follow by Proposition \ref{roma-30}$(2)$, since $x\ra (z\Cap y)\le_Q x\ra y$ and 
$(z\ra x)\ra y\le_Q x\ra y$, respectively. 
\end{proof}

\begin{corollary} \label{roma-90} If $X$ is a quantum-Wajsberg algebra, then $\le_Q$ is a partial order on $X$. 
\end{corollary}
\begin{proof}
It follows by Propositions \ref{qbe-20}$(2)$ and \ref{roma-80}$(7)$. 
\end{proof}


\begin{proposition} \label{roma-100} Let $X$ be an IOM algebra. The following hold for all $x,y,z\in X$: \\
$(1)$ $(z\Cap x)\ra (y\Cap x)=(z\Cap x)\ra y;$ \\
$(2)$ $(x\ra y)^*\Cap x=(x\ra y)^*;$ \\ 
$(3)$ $(x\Cap y)\Cap y=x\Cap y;$ \\
$(4)$ $x\ra (y\ra z)=(x\odot y)\ra z;$ \\
$(5)$ $x\le_Q y\ra z$ implies $x\odot y\le z;$ \\
$(6)$ $x\odot y\le z$ implies $x\le y\ra z;$ \\
$(7)$ $(x\ra y)\odot x\le y;$ \\
$(8)$ $x\le_Q y$ implies $x\odot z\le_Q y\odot z$. 
\end{proposition}
\begin{proof}
$(1)$ Applying Proposition \ref{roma-40}$(3)$, we have: \\
$\hspace*{1.00cm}$ $(z\Cap x)\ra (y\Cap x)=((z^*\ra x^*)\ra x^*)^*\ra (y\Cap x)=(y\Cap x)^*\ra ((z^*\ra x^*)\ra x^*)$ \\
$\hspace*{4.00cm}$ $=(z^*\ra x^*)\ra ((y\Cap x)^*\ra x^*)$ \\
$\hspace*{4.00cm}$ $=(z^*\ra x^*)\ra (x\ra (y\Cap x))=(z^*\ra x^*)\ra (x\ra y)$ \\
$\hspace*{4.00cm}$ $=(z^*\ra x^*)\ra (y^*\ra x^*)=y^*\ra ((z^*\ra x^*)\ra x^*)$ \\
$\hspace*{4.00cm}$ $=((z^*\ra x^*)\ra x^*)^*\ra y=(z\Cap x)\ra y$. \\
$(2)$ Using Proposition \ref{roma-40}$(3)$, we get: \\
$\hspace*{2.00cm}$ $(x\ra y)^*\Cap x=(((x\ra y)\ra x^*)\ra x^*)^*=(((y^*\ra x^*)\ra x^*)\ra x^*)^*$ \\
$\hspace*{4.35cm}$ $=(x\ra ((y^*\ra x^*)\ra x^*)^*)^*=(x\ra (y\Cap x))^*=(x\ra y)^*$. \\ 
$(3)$ Applying $(2)$ we have: \\
$\hspace*{2.00cm}$ $(x\Cap y)\Cap y=((x^*\ra y^*)\ra y^*)^*\Cap y=(y\ra (x^*\ra y^*)^*)^*\Cap y$ \\
$\hspace*{3.85cm}$ $=(y\ra (x^*\ra y^*)^*)^*=((x^*\ra y^*)\ra y^*)^*=x\Cap y$. \\
$(4)$ We have: \\
$\hspace*{1cm}$ $x\ra (y\ra z)=x\ra (z^*\ra x^*)=z^*\ra (x\ra y^*)=(x\ra y^*)^*\ra z=(x\odot y)\ra z$. \\
$(5)$ $x\le_Q y\ra z$ implies $x\le y\ra z$, so that $x\ra (y\ra z)=1$. 
Hence $(x\odot y)\ra z=1$, that is $x\odot y\le z$. \\
$(6)$ From $x\odot y\le z$ we have $(x\odot y)\ra z=1$, that is $(x\ra y^*)^*\ra z=1$. 
It follows that $z^*\ra (x\ra y^*)=1$, so that $x\ra (z^*\ra y^*)=1$, hence $x\ra (y\ra z)=1$ and so 
$x\le y\ra z$. \\
$(7)$ Since $x\ra y\le_Q x\ra y$, by $(5)$ we get $(x\ra y)\odot x\le y$. \\
$(8)$ From $x\le_Q y$ we get $y\ra z^*\le_Q x\ra z^*$ and $(x\ra z^*)^*\le_Q (y\ra z^*)^*$, that is 
$x\odot z\le_Q y\odot z$.
\end{proof}

The next example is derived from \cite[Ex. 6.9]{Ior31}. 

\begin{example} \label{roma-110} 
Let $X=\{0,a,b,c,1\}$ and let $(X,\ra,0,1)$ be the involutive BE algebra with $\ra$ and the corresponding 
operation $\Cap$ given in the following tables:  
\[
\begin{array}{c|ccccccc}
\ra & 0 & a & b & c & 1 \\ \hline
0   & 1 & 1 & 1 & 1 & 1 \\ 
a   & b & 1 & b & 1 & 1 \\ 
b   & a & 1 & 1 & 1 & 1 \\ 
c   & c & 1 & 1 & 1 & 1 \\
1   & 0 & a & b & c & 1  
\end{array}
\hspace{10mm}
\begin{array}{c|ccccccc}
\Cap & 0 & a & b & c & 1 \\ \hline
0    & 0 & 0 & 0 & 0 & 0 \\ 
a    & 0 & a & b & c & a \\ 
b    & 0 & 0 & b & c & b \\ 
c    & 0 & a & b & c & c \\
1    & 0 & a & b & c & 1 
\end{array}
\]
Then $(X,\ra,0,1)$ is an implicative-orthomodular algebra.  
We can see that $a=b\ra (b\Cap a)\neq 1=b\ra a$, hence axiom $(QW_1)$ is not satisfied.  
It follows that $X$ is not a quantum-Wajsberg algebra. 
\end{example}

\begin{theorem} \label{roma-120} An IOM algebra $(X,\ra,0,1)$ is a QW algebra if and only if it satisfies the following condition for all $x,y\in X$: \\ 
$(IOM_2)$ $(x\Cap y)\ra (y\Cap x)=1$. 
\end{theorem}
\begin{proof}
If $X$ is a QW algebra, using $(QW_1)$ we have: \\
$\hspace*{1.00cm}$ $(x\Cap y)\ra (y\Cap x)=(y\Cap x)^*\ra (x\Cap y)^*=(y\Cap x)^*\ra ((x^*\ra y^*)\ra y^*)$ \\ 
$\hspace*{4.00cm}$ $=(x^*\ra y^*)\ra ((y\Cap x)^*\ra y^*)=(y\ra x)\ra (y\ra (y\Cap x))$ \\
$\hspace*{4.00cm}$ $=(y\ra x)\ra (y\ra x)=1$, \\
hence condition $(IOM_2)$ is verified. 
Conversely, assume that $X$ is an IOM algebra satisfying condition $(IOM_2)$. 
Then we have: \\
$\hspace*{2.00cm}$ $1=(y\Cap x)\ra (x\Cap y)=(x\Cap y)^*\ra (y\Cap x)^*$ \\
$\hspace*{2.30cm}$ $=(x\Cap y)^*\ra ((y^*\ra x^*)\ra x^*)=(y^*\ra x^*)\ra ((x\Cap y)^*\ra x^*)$ \\
$\hspace*{2.30cm}$ $=(x\ra y)\ra (x\ra (x\Cap y))$, \\ 
hence $x\ra y\le x\ra (x\Cap y)$. 
On the other hand, $x\Cap y\le_Q y$ implies $x\ra (x\Cap y)\le_Q x\ra y$, and applying Proposition \ref{roma-40}$(7)$ 
we get $x\ra (x\Cap y)=x\ra y$, so that $X$ satisfies condition $(QW_1)$. 
Since by Theorem \ref{roma-50} condition $(QW_2)$ is also satisfied, it follows that $X$ is a QW algebra. 
\end{proof}

\begin{remark} \label{roma-130} Condition $(IOM_2)$ is equivalent to the following condition: \\
$(IOM_2^{'})$ $(x\Cup y)\ra (y\Cup x)=1$. \\
Indeed, if condition $(IOM_2)$ is satisfied for all $x,y\in X$, then 
$(x\Cup y)\ra (y\Cup x)=(x^*\Cap y^*)^*\ra (y^*\Cap x^*)^*=(y^*\Cap x^*)\ra (x^*\Cap y^*)=1$, hence condition 
$(IOM_2^{'})$ is satisfied. The converse follows similarly. 
\end{remark}

\begin{remark} \label{roma-140} 
According to Remark \ref{roma-20-10} and Proposition \ref{qbe-80}, any QW algebra is an IOM algebra satisfying 
the following property for all $x,y\in X$: \\ 
$(Prel)$ $(x\ra y)\Cup (y\ra x)=1$ $\hspace*{0.5cm}$ (prelinearity), \\ 
but the converse is not true.  
Indeed, the IOM algebra from Example \ref{roma-110} satisfies the prelinearity property, but it is not a QW algebra. 
\end{remark}

$\vspace*{1mm}$

\section{Filters and deductive systems in implicative-orthomodular algebras}

The ideals in QMV algebras have been introduced and studied  by R. Giuntini and S. Pulmannov\'a in \cite{Giunt7}, 
and these notions were also investigated in \cite{DvPu, DvPu1}. 
We defined and studied in \cite{Ciu79} the filters and deductive systems in quantum-Wajsberg algebras. 
In this section, we define the filters and deductive systems in implicative-orthomodular algebras, we study their 
properties, and we prove that any deductive system is a filter. 
We also define the maximal and strongly maximal filters, and we show that a strongly maximal filter is maximal. 
Furthermore, we define and characterize the commutative deductive systems, and for the case of quantum-Wajsberg algebras, we prove that any deductive system is commutative. 

\begin{definition} \label{foma-10}   
\emph{
A \emph{filter} of $X$ is a nonempty subset $F\subseteq X$ satisfying the folowing conditions: \\
$(F_1)$ $x,y\in F$ implies $x\odot y\in F;$ \\
$(F_2)$ $x\in F$, $y\in X$, $x\le_Q y$ imply $y\in F$.   
}
\end{definition}

Denote by $\mathcal{F}(X)$ the set of all filters of $X$.

\begin{proposition} \label{foma-20}   
A nonempty subset $F\subseteq X$ is a filter of $X$ if and only if it satisfies conditions $(F_1)$ and \\ 
$(F_3)$ $x\in F$, $y\in X$ imply $y\ra x\in F$. 
\end{proposition}
\begin{proof} 

We show that conditions $(F_2)$ and $(F_3)$ are equivalent. \\
Assume that $F$ satisfies condition $(F_2)$, and let $x\in F$ and $y\in X$. 
Since $x\le_Q y\ra x$, applying $(F_2)$ we get $y\ra x\in F$, hence condition $(F_3)$ is satisfied.
Conversely, assume that $F$ satisfies condition $(F_3)$, and let $x\in F$ and $y\in X$ such that $x\le_Q y$. 
By Proposition \ref{roma-30}$(2)$ and $(F_3)$, $y=y\Cup x=(y\ra x)\ra x\in F$, that is condition $(F_2)$ is verified.  
\end{proof}

\begin{definition} \label{foma-30}    
\emph{
A \emph{deductive system} of $X$ is a subset $F\subseteq X$ satisfying the following conditions: \\
$(DS_1)$ $1\in F;$ \\
$(DS_2)$ $x, x\ra y\in F$ implies $y\in F$. 
}
\end{definition}

Denote by $\mathcal{DS}(X)$ the set of all deductive systems of $X$. 
Obviously $\{1\},X\in \mathcal{F}(X), \mathcal{DS}(X)$. 

\begin{proposition} \label{foma-40}     
$\mathcal{DS}(X)\subseteq \mathcal{F}(X)$. 
\end{proposition}
\begin{proof}
Let $F\in \mathcal{DS}(X)$ and let $x,y\in F$. Since by $(DS_1)$, $1\in F$, it follows that $F$ is nonempty. 
By Lemma \ref{qbe-10}$(8)$ we have $x\ra (y\ra z)=(x\ra y^*)^*\ra z$, for all $x,y,z\in X$. 
Taking $z:=(x\ra y^*)^*$, we get $x\ra (y\ra (x\ra y^*)^*)=(x\ra y^*)^*\ra (x\ra y^*)^*=1\in F$. 
Since $x,y\in F$, by $(DS_2)$ we get $x\odot y=(x\ra y^*)^*\in F$, that is $(F_1)$. \\ 
Let $x,y\in F$ such that $x\le_Q y$. By Proposition \ref{qbe-20}$(4)$, we have $x\le y$, that is $x\ra y=1\in F$. 
Since $x\in F$, by $(DS_2)$ we get $y\in F$, hence condition $(F_2)$ is also satisfied. 
It follows that $F\in \mathcal{F}(X)$, hence $\mathcal{DS}(X)\subseteq \mathcal{F}(X)$.  
\end{proof}

We say that $F\in \mathcal{F}(X)$ is \emph{proper} if $F\neq X$. 

\begin{proposition} \label{foma-50}    
Let $F$ be a subset $F\subseteq X$. The following are equivalent: \\
$(a)$ $F\in \mathcal{DS}(X);$ \\
$(b)$ $F$ is nonempty and it satisfies conditions $(F_1)$ and \\
$(F_4)$ $x\in F$, $y\in X$, $x\le y$ imply $y\in F;$ \\
$(c)$ $F$ is nonempty and it satisfies conditions $(F_1)$ and \\
$(F_5)$ $x\in F$, $y\in X$ imply $x\Cup y\in F$. 
\end{proposition}
\begin{proof}
$(a)\Rightarrow (b)$ Since $F\in \mathcal{DS}(X)$, by Proposition \ref{foma-40}, $F$ satisfies 
condition $(F_1)$. 
If $x\in F$, $y\in X$ such that $x\le y$, then $x\ra y=1\in F$. 
It follows that $y\in F$, so that condition $(F_4)$ is satisfied. \\
$(b)\Rightarrow (a)$ Since $F$ is nonempty, $1\in F$. 
Let $x\in F$, $y\in X$ such that $x\ra y\in F$. 
By $(F_1)$ we get $x\odot (x\ra y)\in F$, and applying Proposition \ref{qbe-120}$(4)$ we get $x\odot (x\ra y)\le y$. 
Using $(F_4)$ we have $y\in F$, thus $F\in \mathcal{DS}(X)$. \\ 
$(b)\Rightarrow (c)$ We have $x\ra (x\Cup y)=x\ra ((x\ra y)\ra y)=(x\ra y)\ra (x\ra y)=1$, hence 
$x\le x\Cup y$. By $(F_4)$ we get $x\Cup y\in F$, so that $F$ satisfies condition $(F_5)$. \\
$(c)\Rightarrow (b)$ Suppose that $x\in F$ implies $x\Cup y\in F$, and let $y\in X$ such that $x\le y$, 
thus $x\ra y=1$. We have $x\Cup y=(x\ra y)\ra y=y$ and $x\Cup y\in F$, hence $y\in F$, 
that is $(F_4)$ is verified. 
\end{proof}

\begin{lemma} \label{foma-60}    
If $F\in \mathcal{F}(X)$ or $F\in \mathcal{DS}(X)$, then $F$ is a subalgebra of $X$. 
\end{lemma}
\begin{proof}
Let $F\in \mathcal{F}(X)$ and let $x,y\in F$. We have $x\le_Q y\ra x$ and $y\le_Q x\ra y$, hence by $(F_2)$ 
we get $x\ra y, y\ra x\in F$. 
If $F\in \mathcal{DS}(X)$ and $x,y\in F$, from $x\le_Q y\ra x$, $y\le_Q x\ra y$ we have $x\le y\ra x$ and 
$y\le x\ra y$, that is $x\ra (y\ra x)=y\ra (x\ra y)=1\in F$. Since $x,y\in F$, we get $x\ra y, y\ra x\in F$. 
Thus $F$ is a subalgebra of $X$. 
\end{proof}

\begin{example} \label{foma-70} 
Let $(X,\ra,0,1)$ be the implicative-orthomodular algebra from Example \ref{roma-110}.  
Then $\mathcal{DS}(X)=\{\{1\},X\}$ and $\mathcal{F}(X)=\{\{1\},\{a,1\},X\}$. 
\end{example}

If $Y\subseteq X$, the smallest filter of $X$ containing $Y$ (i.e. the intersection of all filters $F$ of $X$ 
such that $Y\subseteq F$) is called the \emph{filter generated} by $Y$ and 
it is denoted by $[Y)$. If $Y=\{x\}$ we write $[x)$ instead of $[\{x\})$ and $[x)$ is called a 
\emph{principal filter} of $X$. We can easily prove that: \\
$\hspace*{1cm}$ $[Y)=\{y\in X\mid y\ge_Q y_1\odot y_2\odot \cdots \odot y_n$, for some $n\ge 1$ and 
$y_1,y_2,\dots,y_n\in Y\}$ and \\
$\hspace*{1cm}$ $[x)=\{y\in X\mid y\ge_Q x^n$, for some $n\ge 1\}$, for any $x\in X$. \\
If $F\in \mathcal{F}(X)$ and $x\in X$, we also can show that: \\
$\hspace*{1cm}$ $F(x)=[F\cup \{x\})=\{y\in X\mid y\ge_Q f\odot x^n$, for some $n\ge 1$ and $f\in F\}$. 
Obviously, if $x\in F$, then $F(x)=F$. 

\begin{definition} \label{foma-80}     
\emph{
A filter $F\in \mathcal{F}(X)$ is said to be: \\
$(1)$ \emph{maximal} if it is proper and it is not contained in any other proper filter of $X;$ \\
$(2)$ \emph{strongly maximal} if, for all $x\in X$, $x\notin F$, there exists $n\ge 1$ such that $(x^n)^*\in F$. 
}
\end{definition}

\begin{proposition} \label{foma-90}   
If $F\in \mathcal{F}(X)$, the following are equivalent: \\
$(a)$ $F$ is maximal; \\
$(b)$ for any $x\notin F$, there exist $f\in F$ and $n\ge 1$ such that $f\odot x^n=0$. 
\end{proposition}
\begin{proof}
$(a) \Rightarrow (b)$ Since $F$ is maximal and $x\notin F$, then $F(x)=[F\cup \{x\})=X$, hence $0\in F(x)$. 
It follows that there exist $f\in F$ and $n\ge 1$ such that $f\odot x^n\le_Q 0$, that is $f\odot x^n= 0$. \\
$(b) \Rightarrow (a)$ Let $F^{'}$ be a proper filter of $X$ such that $F\subsetneq  F^{'}$. 
Then there exists $x\in F^{'}$ such that $x\notin F$. 
It follows that there exist $f\in F$ and $n\ge 1$ such that $f\odot x^n=0$. Since $f,x\in F^{'}$ we get 
$0\in F^{'}$, that is $F^{'}=X$. Hence $F$ is a maximal filter of $X$. 
\end{proof}

\begin{proposition} \label{foma-100}     
Every strongly maximal filter of $X$ is maximal.
\end{proposition}
\begin{proof}
Let $F$ be a strongly maximal filter of $X$ and let $F^{'}\in \mathcal{F}(X)$ such that $F\subseteq F^{'}$. 
Suppose that there exists $x\in F^{'}$, $x\notin F$. It follows that there exists $n\ge 1$ such that $(x^n)^*\in F$. 
Since $x\in F^{'}$, we have $x^n\in F^{'}$, hence $0=x^n\odot (x^n)^*\in F^{'}$. 
Thus $F^{'}=X$, that is $F$ is maximal.
\end{proof}

\begin{proposition} \label{foma-110}    
If $\mathcal{F}(X)=\mathcal{DS}(X)$, then every maximal filter of $X$ is strongly maximal.  
\end{proposition}
\begin{proof}
Let $X$ be a maximal filter of $X$ and let $x\in X$ such that $x\notin F$. 
Then $F(x)=X$, hence $0\in F(x)$. 
It follows that there exist $n\ge 1$ and $f\in F$ such that $f\odot x^n\le_Q 0$, so that $f\odot x^n\le 0$. 
By Proposition \ref{qbe-120}$(3)$, we have $f\le x^n\ra 0=(x^n)^*$. 
Since $F\in \mathcal{DS}(X)$, by $(F_4)$ we get $(x^n)^*\in F$, hence $F$ is strongly maximal. 
\end{proof}

\begin{definition} \label{foma-120}    
\emph{
A filter $F\in \mathcal{F}(X)$ is called \emph{commutative} if it satisfies the following condition 
for all $x,y\in X$: \\
$(CF)$ $y\ra x\in F$ implies $(x\Cup y)\ra x\in F$. 
}
\end{definition}

Since $\mathcal{DS}(X)\subseteq \mathcal{F}(X)$ (Proposition \ref{foma-40}), this definition is also valid for the case of deductive systems. 
Denote by $\mathcal{F}_c(X)$ the set of all commutative filters and by $\mathcal{DS}_c(X)$ the set of all 
commutative deductive systems of $X$. 

\begin{proposition} \label{foma-130}  
Let $F\subseteq X$. Then $F\in \mathcal{DS}_c(X)$ if and only if it satisfies the following conditions 
for all $x,y,z\in X$: \\
$(1)$ $1\in F;$ \\
$(2)$ $z\ra (y\ra x)\in F$ and $z\in F$ imply $(x\Cup y)\ra x\in F$.   
\end{proposition}
\begin{proof}
Let $F\in \mathcal{DS}_c(X)$. Condition $(1)$ is satisfied by $(DS_1)$. 
Let $x,y,z\in F$ such that $z\ra (y\ra x)\in F$ and $z\in F$. 
By $(DS_2)$, we have $y\ra x\in F$, so that by $(CF)$, $(x\Cup y)\ra x\in F$. 
Conversely, let $F\subseteq X$ satisfying conditions $(1)$ and $(2)$. Obviously $1\in F$. 
Let $x,y\in F$ such that $x, x\ra y\in F$. Since $x\ra (1\ra y)=x\ra y\in F$, by $(2)$ we have $y=(y\Cup 1)\ra y\in F$. 
It follows that $F\in \mathcal{DS}(X)$. 
Let $x,y\in X$ such that $y\ra x\in F$. Since $1\ra (y\ra x)\in F$ and $1\in F$, by $(2)$ we get 
$(x\Cup y)\ra x\in F$, that is $(CF)$. Hence $F\in \mathcal{DS}_c(X)$. 
\end{proof}

\begin{proposition} \label{foma-140}  
Let $F\in \mathcal{DS}_c(X)$ and $E\in \mathcal{DS}(X)$ such that $F\subseteq E$. Then $E\in \mathcal{DS}_c(X)$.  
\end{proposition}
\begin{proof}
Let $x,y\in X$ such that $u=y\ra x\in E$. 
It follows that $y\ra (u\ra x)=y\ra ((y\ra x)\ra x)=(y\ra x)\ra (y\ra x)=1\in F$. 
Since $F$ is commutative, $((u\ra x)\Cup y)\ra (u\ra x)\in F$, and $F\subseteq E$ implies 
$((u\ra x)\Cup y)\ra (u\ra x)\in E$. 
It follows that $u\ra (((u\ra x)\Cup y)\ra x)\in E$. 
Since $u\in E$, using $(DS_2)$ we get $((u\ra x)\Cup y)\ra x\in E$.   
From $x\le_Q u\ra x$, we have $(u\ra x)\ra y\le_Q x\ra y$ and $(x\ra y)\ra y\le_Q ((u\ra x)\ra y)\ra y$, 
that is $x\Cup y\le_Q (u\ra x)\Cup y$. 
It follows that $((u\ra x)\Cup y)\ra x\le_Q (x\Cup y)\ra x$, hence $(x\Cup y)\ra x\in E$, so that by $(F_2)$,  
$E\in \mathcal{DS}_c(X)$. 
\end{proof}

\begin{corollary} \label{foma-150} 
$\{1\}\in \mathcal{DS}_c(X)$ if and only if $\mathcal{DS}(X)=\mathcal{DS}_c(X)$.  
\end{corollary}

\begin{proposition} \label{foma-160} If $X$ is a QW algebra, then $\mathcal{DS}(X)=\mathcal{DS}_c(X)$.  
\end{proposition}
\begin{proof}
Applying axiom $(QW_1)$, we get $(x\Cup y)\ra x=x^*\ra (x\Cup y)^*=x^*\ra (x^*\Cap y^*)=x^*\ra y^*=y\ra x$. 
Hence, for any $F\in \mathcal{DS}(X)$ we have $(x\Cup y)\ra x\in F$ if and only if $y\ra x\in F$, that is 
$F\in \mathcal{DS}_c(X)$. 
\end{proof}

\begin{example} \label{foma-170} 
\emph{
Let $(X,\ra,0,1)$ be an implicative-orthomodular algebra. 
A \emph{Bosbach state} on $X$ is a map $s:X\longrightarrow [0,1]$ satisfying the following conditions 
for all $x,y\in X$: \\
$(bs_1)$ $s(0)=0$ and $s(1)=1;$ \\
$(bs_2)$ $s(x)+s(x\ra y)=s(y)+ s(y\ra x)$. \\
Denote by $\mathcal{BS}(X)$ the set of all Bosbach states on $X$. 
For any $s\in \mathcal{BS}(X)$, the set $\Ker(s)=\{x\in X\mid s(x)=1\}$ is called the \emph{kernel} of $s$. 
If $x,x\ra y\in \Ker(s)$, by $(bs_2)$ we get $s(y)=s(y\ra x)=1$, hence $y\in \Ker(s)$. 
Since by $(bs_1)$, $1\in \Ker(s)$, it follows that $\Ker(s)\in \mathcal{DS}(X)$. 
Let $x,y\in X$ such that $y\ra x\in \Ker(s)$, that is $s(y\ra x)=1$. 
Using $(bs_2)$, we get $s(x\ra y)=1-s(x)+s(y)$ and $s(x)+s(x\ra (x\Cup y))=s(x\Cup y)+s((x\Cup y)\ra x)$. 
Since $x\ra (x\Cup y)=x\ra ((x\ra y)\ra y)=(x\ra y)\ra (x\ra y)=1$, we have $s(x\ra (x\Cup y))=1$. 
Moreover from $s(x\ra y)+s((x\ra y)\ra y)=s(y)+s(y\ra (x\ra y))$, we get 
$s(x\Cup y)=1-s(x\ra y)+s(y)=1-1+s(x)-s(y)+s(y)=s(x)$. 
It follows that $s((x\Cup y)\ra x)=s(x)+1-s(x\Cup y)=s(x)+1-s(x)=1$, hence $(x\Cup y)\ra x\in \Ker(s)$, that is 
$\Ker(s)\in \mathcal{DS}_c(X)$. 
}
\end{example}

$\vspace*{1mm}$

\section{Congruences in implicative-orthomodular algebras}

In this section, we introduce the notion of congruences determined by deductive systems of an 
implicative-orthomodular algebra $X$. 
We prove that any deductive system of $X$ induces a congruence on $X$, and conversely, for any congrence on $X$ 
we can define a deductive system. 
Furthermore, we define the quotient implicative-orthomodular algebra with respect to the congruence induced by a 
deductive system.
One of the main results consists of proving that a deductive system is commutative if and only if all deductive 
systems of the corresponding quotient algebra are commutative. 

\begin{definition} \label{cobe-10} Let $F\in \mathcal{F}(X)$ and let $x,y\in X$. Then $x\equiv_F y$ if and only if 
there exists $f\in X$ such that $x,y\le_Q f$ and $f\ra x, f\ra y\in F$. 
\end{definition}

\begin{proposition} \label{cobe-20} If $F\in \mathcal{DS}(X)$, the following are equivalent for all $x,y\in X$: \\
$(a)$ $x\equiv_F y;$ \\ 
$(b)$ there exist $f_1, f_2 \in F$ such that $x\le_Q f_1$, $y\le_Q f_2$ and $f_1\ra x=f_2\ra y;$ \\   
$(c)$ there exist $f_1, f_2 \in F$ such that $x\le_Q f_2\ra y$ and $y\le_Q f_1\ra x$. 
\end{proposition}
\begin{proof}
$(a)\Rightarrow (b)$ Suppose that $x\equiv_F y$, that is there exists $f\in X$ such that $x,y\le_Q f$ 
and $f\ra x, f\ra y\in F$. By Proposition \ref{roma-30}$(2)$, $f=f\Cup x=f\Cup y$, and  
taking $f_1:=f\ra x$, $f_2:=f\ra y$, we have $x\le_Q f_1$ and $y\le_Q f_2$.   
Moreover $f_1\ra x=(f\ra x)\ra x=f\Cup x=f=f\Cup y=(f\ra y)\ra y=f_2\ra y$, hence condition $(b)$ is satisfied. \\
$(b)\Rightarrow (a)$ Assume that there exist $f_1, f_2 \in F$ such that $x\le_Q f_1$, $y\le_Q f_2$ 
and $f_1\ra x=f_2\ra y$. 
Taking $f:=f_1\ra x=f_2\ra y$, we have $x, y\le_Q f$. 
Moreover $f\ra x=(f_1\ra x)\ra x=f_1\Cup x=f_1\in F$ (since $x\le_Q f_1)$ and 
$f\ra y=(f_2\ra y)\ra y=f_2\Cup y=f_2\in F$ (since $y\le_Q f_2$). 
It follows that $x\equiv_F y$. \\
$(b)\Rightarrow (c)$ With $f_1, f_2$ from $(b)$ we have $x\le_Q f_1\ra x=f_2\ra y$ and 
$y\le_Q f_2\ra y=f_1\ra x$, hence condition $(c)$ is satisfied. \\
$(c)\Rightarrow (b)$ Assume that there exist $f_1, f_2 \in F$ such that $x\le_Q f_2\ra y$ and $y\le_Q f_1\ra x$.  
Taking $g_1:=f_1\Cup x$ and $g_2:=(f_1\ra x)\ra y$, we have $x\le_Q g_1$ and $y\le_Q g_2$.  
Since $F\in \mathcal{DS}(X)$ and $f_1\in F$, by Proposition \ref{foma-50} we get $g_1=f_1\Cup x\in F$. 
From $f_2\ra y\ge_Q x$ we have 
$f_2\ra g_2=f_2\ra ((f_1\ra x)\ra y)=(f_1\ra x)\ra (f_2\ra y)\ge_Q (f_1\ra x)\ra x=f_1\Cup x=g_1\in F$. 
Thus $f_2\ra g_2\in F$.  
Since $F\in \mathcal{DS}(X)$ and $f_2, f_2\ra g_2\in F$, we get $g_2 \in F$. 
Moreover $g_1\ra x=(f_1\Cup x)\ra x=f_1\ra x$ (by Proposition \ref{roma-40}$(9)$) and 
$g_2\ra y=((f_1\ra x)\ra y)\ra y=(f_1\ra x)\Cup y=f_1 \ra x$, since $y\le_Q f_1\ra x$.  
It follows that $g_1\ra x=g_2\ra y=f_1\ra x$, so that condition $(b)$ is verified with 
$g_1, g_2$ instead of $f_1, f_2$. 
\end{proof}

\begin{theorem} \label{cobe-30} Let $F\in \mathcal{DS}(X)$. The relation $\equiv_F$ is an equivalence relation 
on $X$. 
\end{theorem}
\begin{proof}
We can easily show that the relation $\equiv_F$ is reflexive and symmetric. \\
Let $x,y,z\in F$ such that $x\equiv_F y$ and $y\equiv_F z$. We will show that $x\equiv_F z$. 
By definition, there exist $f_1, f_2 \in X$ such that $x,y\le_Q f_1$, $f_1\ra x, f_1\ra y\in F$ 
and $y,z\le_Q f_2$, $f_2\ra y, f_2\ra z\in F$. 
If $g_1:=(f_1\ra x)\odot (f_2\ra y)\in F$ and $g_2:=(f_1\ra y)\odot (f_2\ra z)\in F$ we prove that 
$x\le_Q g_2\ra z$ and $z\le_Q g_1\ra x$. 
Since $x\le_Q f_1$ implies $f_1\Cup x=f_1$, $f_1\ge_Q y$ implies 
$(f_2\ra y)\ra f_1\ge_Q (f_2\ra y)\ra y$, and $y\le_Q f_2$ implies $f_2\Cup y=f_2$, we have: \\
$\hspace*{2.00cm}$ $g_1\ra x=x^*\ra g_1^*=x^*\ra ((f_1\ra x)\odot (f_2\ra y))^*$ \\
$\hspace*{3.10cm}$ $=x^*\ra ((f_1\ra x)\ra (f_2\ra y)^*)$ \\
$\hspace*{3.10cm}$ $=x^*\ra ((f_2\ra y)\ra (f_1\ra x)^*)$ \\
$\hspace*{3.10cm}$ $=(f_2\ra y)\ra (x^*\ra (f_1\ra x)^*)$ \\
$\hspace*{3.10cm}$ $=(f_2\ra y)\ra ((f_1\ra x)\ra x)$ \\
$\hspace*{3.10cm}$ $=(f_2\ra y)\ra (f_1\Cup x)$ \\
$\hspace*{3.10cm}$ $=(f_2\ra y)\ra f_1\ge_Q (f_2\ra y)\ra y$ \\
$\hspace*{3.10cm}$ $=f_2\Cup y=f_2\ge_Q z$. \\
Similarly, since $z\le_Q f_2$ implies $f_2\Cup z=f_2$, $f_2\ge_Q y$ implies 
$(f_1\ra y)\ra f_2\ge_Q (f_1\ra y)\ra y$, and $y\le_Q f_1$ implies $f_1\Cup y=f_1$, we get: \\
$\hspace*{2.00cm}$ $g_2\ra z=z^*\ra g_2^*=z^*\ra ((f_1\ra y)\odot (f_2\ra z))^*$ \\
$\hspace*{3.10cm}$ $=z^*\ra ((f_1\ra y)\ra (f_2\ra z)^*)$ \\
$\hspace*{3.10cm}$ $=(f_1\ra y)\ra (z^*\ra (f_2\ra z)^*)$ \\
$\hspace*{3.10cm}$ $=(f_1\ra y)\ra ((f_2\ra z)\ra z)$ \\
$\hspace*{3.10cm}$ $=(f_1\ra y)\ra (f_2\Cup z)$ \\
$\hspace*{3.10cm}$ $=(f_1\ra y)\ra f_2\ge_Q (f_1\ra y)\ra y$ \\
$\hspace*{3.10cm}$ $=f_1\Cup y=f_1\ge_Q x$. \\
By Proposition \ref{cobe-20}$(c)$, it follows that $x\equiv_F z$, that is $\equiv_F$ is transitive, so that 
it is an equivalence relation on $X$. 
\end{proof}

\begin{proposition} \label{cobe-40} Let $F\in \mathcal{DS}(X)$ and let $x,y\in X$. 
If $x\equiv_F y$, then $x^*\equiv_F y^*$.  
\end{proposition}
\begin{proof}
Since $x\equiv_F y$, there exists $f\in X$ such that $x,y\le_Q f$ and $f\ra x, f\ra y\in F$. 
It follows that $x^*\ge_Q f^*$ and $y^*\ge_Q f^*$. 
Denote $g:=x\ra (f\odot y^*)$, so that $g\ge_Q x^*$. 
Obviously $y\le_Q f$ implies $f=f\Cup y$, and $f\odot y^*\le_Q y^*$ 
implies $y^*=y^*\Cup (f\odot y^*)$. Taking into consideration that $x\le_Q f$, we get: \\
$\hspace*{2.00cm}$ $g=x\ra (f\odot y^*)\ge_Q f\ra (f\odot y^*)=(f\Cup y)\ra (f\odot y^*)$ \\
$\hspace*{2.30cm}$ $=((f\ra y)\ra y)\ra (f\odot y^*)=(y^*\ra (f\ra y)^*)\ra (f\odot y^*)$ \\
$\hspace*{2.30cm}$ $=(y^*\ra (f\odot y^*))\ra (f\odot y^*)=y^*\Cup (f\odot y^*)=y^*$. \\
Hence $g\ge_Q y^*$. 
Moreover: \\
$\hspace*{2.00cm}$ $g\ra x^*=(x\ra (f\odot y^*))\ra x^*=((f\odot y^*)^*\ra x^*)\ra x^*$ \\
$\hspace*{3.30cm}$ $=(f\odot y^*)^*\Cup x^*=(f\ra y)\Cup x^*\in F$, \\
since $f\ra y\in F$, and $F\in \mathcal{DS}(X)$. \\
By Proposition \ref{qbe-30}$(6)$,$(5)$, $x\ra (f\odot y^*)=((f\ra y)\odot x)^*$ and $x\le_Q f$ implies 
$x=(f\ra x)\odot f$. 
Since $f\ge_Q y$ we get: \\  
$\hspace*{2.00cm}$ $(f\ra x)\odot (f\ra y)\odot f \ge_Q (f\ra x)\odot (f\ra y)\odot y$ and \\ 
$\hspace*{2.00cm}$ $((f\ra x)\odot (f\ra y)\odot f)^*\ra y^* 
                   \ge_Q ((f\ra x)\odot (f\ra y)\odot y)^*\ra y^*$. \\
It follows that: \\
$\hspace*{2.00cm}$ $g\ra y^*=(x\ra (f\odot y^*))\ra y^*=((f\ra y)\odot x)^*\ra y^*$ \\
$\hspace*{3.30cm}$ $=((f\ra y)\odot (f\ra x)\odot f)^*\ra y^*$ \\
$\hspace*{3.30cm}$ $\ge_Q ((f\ra x)\odot (f\ra y)\odot y)^*\ra y^*$ \\
$\hspace*{3.30cm}$ $=(((f\ra x)\odot (f\ra y))\ra y^*)\ra y^*$ \\
$\hspace*{3.30cm}$ $=((f\ra x)\odot (f\ra y))\Cup y^*\in F$, \\
since $(f\ra x)\odot (f\ra y)\in F$, and $F\in \mathcal{DS}(X)$. 
We proved that there exists $g\in X$ such that $x^*, y^*\le_Q g$, and 
$g\ra x^*, g\ra y^*\in F$, hence $x^*\equiv_F y^*$. 
\end{proof}

\begin{proposition} \label{cobe-50} Let $F\in \mathcal{DS}(X)$ and let $x,y\in X$. 
If $x\equiv_F y$ and $u\equiv_F v$, then $x\odot u\equiv_F y\odot v$.  
\end{proposition}
\begin{proof}
Since $x\equiv_F y$ and $u\equiv_F v$, there exist $f, g\in X$ such that $x,y\le_Q f$, $u,v\le_Q g$ and 
$f\ra x, f\ra y, g\ra u, g\ra v\in F$. 
Denote $h:=f\odot g$, and obviously $x\odot u, y\odot v\le_Q h$. 
Since the operation $\odot$ is commutative and associative, applying Proposition \ref{qbe-30}$(5)$, we get: \\ 
$\hspace*{1.50cm}$ $((f\ra x)\odot (g\ra u))\ra (h\ra (x\odot u))=$ \\
$\hspace*{4.00cm}$ $=((f\ra x)\odot (g\ra u))\ra ((f\odot g)\ra (x\odot u))$ \\
$\hspace*{4.00cm}$ $=(f\ra x)\odot (g\ra u)\ra ((f\odot g)\odot (x\odot u)^*)^*$ \\
$\hspace*{4.00cm}$ $=[(f\ra x)\odot (g\ra u)\odot (f\odot g)\odot (x\odot u)^*]^*$ \\
$\hspace*{4.00cm}$ $=[((f\ra x)\odot f)\odot ((g\ra u)\odot g)\odot (x\odot u)^*]^*$ \\
$\hspace*{4.00cm}$ $=[(x\odot u)\odot (x\odot u)^*]^*=(x\odot u)\ra (x\odot u)=1\in F$. \\                   
Taking into consideration that $(f\ra x)\odot (g\ra u)\in F$ and $F\in \mathcal{DS}(X)$, we get 
$h\ra (x\odot u)\in F$. 
We can prove similarly that $h\ra (y\odot v)\in F$, hence $x\odot u\equiv_F y\odot v$.  
\end{proof}

\begin{corollary} \label{cobe-60} For any $F\in \mathcal{DS}(X)$, $x\equiv_F y$ and $u\equiv_F v$ imply 
$x\ra u\equiv_F y\ra v$, $x\Cup u\equiv_F y\Cup v$ and $x\Cap u\equiv_F y\Cap v$. 
\end{corollary} 
\begin{proof}
Applying Proposition \ref{cobe-40} we have $u^*\equiv_F v^*$, while by Proposition \ref{cobe-50}, 
$x\equiv_F y$ and $u^*\equiv_F v^*$ imply $x\odot u^*\equiv_F y\odot v^*$. 
Using again Proposition \ref{cobe-40}, we get $(x\odot u^*)^*\equiv_F (y\odot v^*)^*$, 
that is $x\ra u\equiv_F y\ra v$. 
It follows that $(x\ra u)\ra u \equiv_F (y\ra v)\ra v$, that is $x\Cup u\equiv_F y\Cup v$. 
Since $x\Cap u=(x^*\Cup u^*)^*$ and $y\Cap v=(y^*\Cup v^*)^*$, using Proposition \ref{cobe-40} we get 
$x\Cap u\equiv_F y\Cap v$. 
\end{proof}

\begin{theorem} \label{cobe-70} For any $F\in \mathcal{DS}(X)$, there exists a congruence $\equiv_F$ 
on X such that $\{x\in X\mid x\equiv_F 1\}=F$. Conversely, given a congruence $\equiv$ on $X$, then 
$\{x\in X\mid x\equiv 1\}\in \mathcal{DS}(X)$. 
\end{theorem}
\begin{proof}
Let $F\in \mathcal{DS}(X)$. According to Theorem \ref{cobe-30}, $\equiv_F$ is an equivalence relation on $X$. 
Since by Corollary \ref{cobe-60}, $x\equiv_F y$ and $u\equiv_F v$ imply $x\ra u\equiv_F y\ra v$, it follows that  
$\equiv_F$ is a congruence on $X$. 
If $x\in \{x\in X\mid x\equiv_F 1\}$, then by Proposition \ref{cobe-20}$(b)$, there exist $f, g\in F$ such that 
$x\le_Q f$, $1\le_Q g$ and $f\ra x=g \ra 1=1$. 
Since $x\le_Q f$, we have $f\Cup x=f$, hence 
$x=1\ra x=(f\ra x)\ra x=f\Cup x=f\in F$. 
Thus $\{x\in X\mid x\equiv_F 1\}\subseteq F$. 
Let $x\in F$. For $f:=x$, $g:=1$ we have $f,g \in F$, $x\le_Q f$, $1\le_Q g$, and $f\ra x=g\ra 1=1$. 
Hence $x\equiv_F 1$, so that $F\subseteq \{x\in X\mid x\equiv_F 1\}$. 
It follows that $\{x\in X\mid x\equiv_F 1\}= F$. \\ 
Conversely, let $\equiv$ be a congruence on $X$. For all $x,y,u,v\in X$ such that $x\equiv y$ and $u\equiv v$, 
we have $x\ra u\equiv y\ra v$ and $x^*\equiv y^*$.   
Let $F=\{x\in X\mid x\equiv 1\}$ and let $x,y\in F$. 
We can easily see that $x\odot y=(x\ra y^*)^*\equiv 1$, hence $x\odot y\in F$, that is condition $(F_1)$. 
Moreover, for any $x\in F$ and $y\in X$, from $x\equiv 1$ we get $x\Cup y=(x\ra y)\ra y\equiv y\ra y=1$. 
It follows that $x\Cup y\equiv 1$, hence $x\Cup y\in F$, thus condition $(F_5)$ is also verified. 
Finally, by Proposition \ref{foma-50}, we conclude that $F \in \mathcal{DS}(X)$.
\end{proof}

\noindent
The \emph{quotient IOM algebra} induced by $\equiv_F$ is denoted by $X/F$, while the equivalence class of $x\in X$ 
is denoted by $x/F$. 

\begin{proposition} \label{cobe-80} $F\in \mathcal{DS}_c(X)$ if and only if $\mathcal{DS}(X/F)=\mathcal{DS}_c(X/F)$.  
\end{proposition}
\begin{proof}
Let $F\in \mathcal{DS}_c(X)$ and let $x,y\in X$ such that $y/F\ra x/F \in \{1/F\}$, that is $y\ra x\in F$. 
Since $F$ is commutative, we have $(x\Cup y)\ra x\in F$. 
Thus $(x/F\Cup y/F)\ra x/F\in \{1/F\}$, so that $\{1/F\}\in \mathcal{DS}_c(X/F)$, and by Corollary \ref{foma-150}, 
we get $\mathcal{DS}(X/F)=\mathcal{DS}_c(X/F)$. 
Conversely assume that $\mathcal{DS}(X/F)=\mathcal{DS}_c(X/F)$, so that $\{1/F\}\in \mathcal{DS}_c(X/F)$. 
Let $x,y\in X$ such that $y\ra x\in F$, that is $y/F\ra x/F =1/F$. 
Since $1/F$ is commutative, it follows that $(x/F\Cup y/F)\ra x/F\in 1/F$. 
Hence $((x\Cup y)\ra x)/F= 1/F$, thus $(x\Cup y)\ra x\in F$, that is $F\in \mathcal{DS}_c(X)$. 
\end{proof}

$\vspace*{5mm}$

\section{Concluding remarks}

As we mentioned, the orthomodular algebras were introduced and studied by A. Iorgulescu in \cite{Ior31} as generalizations of quantum-MV algebras (QMV algebras for short). 
We redefined in \cite{Ciu78} the quantum-MV algebras as involutive BE algebras, by introducing the notion of quantum-Wajsberg algebras and we studied these structures.  
To be able to compare certain quantum structures, it is preferable that they have the same 
signature, which would also allow possible generalizations of the quantum structures.  
For this reason, in this paper, we redefined the orthomodular algebras, by defining the notion of implicative-orthomodular algebras. 
We characterized these algebras, and gave conditions for implicative-orthomodular algebras to be quantum-Wajsberg algebras. 
We proved that axioms (IOM) and (Pom) are equivalent, so that the implicative-orthomodular algebras are 
equivalent to orthomodular algebras. 
We defined the notions of filters, deductive systems, maximal, strongly maximal and commutative  
deductive systems in IOM algebras, as well as the congruences induced by deductive systems. 
We also defined the quotient IOM algebra with respect to a deductive system, proving that a deductive system is commutative if and only if all deductive systems of the corresponding quotient algebra are commutative. 
The reader can see that the filters correspond to q-ideals, while the deductive systems correspond to 
p-ideals defined in \cite{Giunt7} for the case of QMV algebras. \\
R. Giuntini introduced in \cite{Giunt4} the notion of a commutative center of QMV algebras, as the subset of 
those elements that commute with all its elements. 
The author proved that the commutative center of a QMV algebra is an MV algebra. 
As a further direction of research, one can define and study various centers of implicative-orthomodular algebras. 
R. Giuntini defined also the quasilinear and weakly linear QMV algebras (\cite{Giunt1, Giunt6}) and studied their properties. 
As another topic of research, we could investigate these notions in the case of implicative-orthomodular algebras.

$\vspace*{1mm}$
          
\begin{center}
\sc Acknowledgement 
\end{center}
The author is very grateful to the anonymous referees for their useful remarks and suggestions on the subject that helped improving the presentation.

$\vspace*{1mm}$

\end{document}